\documentclass{article}

\usepackage[english]{babel}

\usepackage[letterpaper,top=2cm,bottom=2cm,left=3cm,right=3cm,marginparwidth=1.75cm]{geometry}

\usepackage{amsmath}
\usepackage{amssymb}
\usepackage{graphicx}
\usepackage[colorlinks=true, allcolors=blue]{hyperref}
\usepackage{xcolor}
\newtheorem{definition}{Definition}
\newtheorem{proposition}{Proposition}
\newtheorem{theorem}{Theorem}
\newtheorem{assumption}{Assumption}
\newtheorem{lemma}{Lemma}
\newtheorem{remark}{Remark}
\newtheorem{fact}{Fact}
\newtheorem{example}{Example}

\newenvironment{proof}[1][\!]{\noindent {\em Proof #1. }}{\hfill $\blacksquare$ \vskip 3pt}

\newcommand{\be}[1]{\begin{equation}\label{#1}}
\newcommand{\ee}{\end{equation}}

\title{Robust output stability and input-to-output stability \newline based on output dissipation}
\author{Antoine Chaillet\footnote{{\tt antoine.chaillet@centralesupelec.fr}, Universit\'e Paris Saclay, CNRS, CentraleSup\'elec, Laboratoire des signaux et syst\`emes, 91190 Gif sur Yvette, France}, Iasson Karafyllis\footnote{{\tt iasonkar@central.ntua.gr}, Mathematics Department, National Technical University of Athens, Greece}, Yuan Wang\footnote{{\tt ywang@fau.edu}, Department of Mathematics and Statistics, Florida Atlantic University, Boca Raton, FL 33431, USA}}

\usepackage{graphicx}          

\usepackage{amsmath}
\usepackage{amssymb}
\usepackage{graphicx}
\usepackage[colorlinks=true, allcolors=blue]{hyperref}
\usepackage{xcolor}

\begin{document}
\newcommand\mR{\mathbb{R}}
\newcommand\R{\mR}
\newcommand\mRp{\mathbb{R}_{\geq 0}}
\def\cL{L^\infty_{\textrm{loc}}}
\def\cK{\mathcal K}
\def\cKL{\mathcal{KL}}
\newcommand{\comment}[1]{}







\maketitle


\vspace{3mm}
\textbf{Abstract.} LaSalle techniques to ensure the convergence of a given output usually fail at guaranteeing uniform convergence time, which induces robustness issues. Recent works have provided extra conditions under which a Lyapunov function that dissipates in terms of only the output guarantees this uniformity. In this paper, we extend these results to systems with inputs in order to establish either robust stability or input-to-output stability. In addition, we show that a recent relaxation of Barb\u{a}lat's lemma is also applicable to systems with inputs. The 
significance of the proposed results is demonstrated in the context of a recent adaptive control scheme.

\section{Introduction}

In a wide range of applications, including adaptive control, observer design, or 
regulation problems, 
the objective is not to control the full state, but only some state variables of interest. For instance, in adaptive control, one usually tries to impose that the plant state converges to the target, whereas the parameter estimation error is not requested to go to zero. In observer design, the relevant state variable is the observation error between the actual state and the reconstructed one. In regulation problems, a primary objective is to guarantee that the tracking error converges to zero.
%
More generally, one often wants that a particular output of interest converges to zero, regardless of the overall behavior of the state. In addition to the convergence of the relevant output, it is often desirable to impose the property of small output transients when the initial state is small. Several tools exist in the literature to address this analysis, including partial stability \cite{VORbook}, stability with respect to two measures  \cite{Movchan60,LAKLEE,TEEPRAconverse}, and output stability \cite{KAJIbook11,SONWANIOS}.

An additional feature, that may look purely technical at first sight, is the uniformity of the convergence rate, meaning that, there is a common bound on the time required for the output to enter a prescribed neighborhood of the origin, uniformly over all initial states and inputs from a bounded set. It is known that this uniformity comes for free for finite-dimensional systems when the considered output is the entire state \cite[Theorem 2]{SONWANTAC}. On the contrary, in the context of output stability, this uniformity is no longer automatic and a system can very well be output stable, with all outputs converging to the origin, without exhibiting a uniform convergence rate over bounded sets of initial states and inputs \cite{ORCHSI20}.

This lack of uniformity may be problematic for two reasons. First, it allows for possible arbitrarily slow transients even when initial states and inputs are in a given prescribed ball. Second, it poses robustness issues, as this uniformity is required in converse Lyapunov techniques for output stability \cite{SONWANIOS-LYA}.

A typical situation in which this uniformity is not guaranteed is when the convergence is established using LaSalle or Barb\u{a}lat techniques. This typically arises when the considered Lyapunov function's dissipation rate involves the output variables only. The example in \cite{ORCHSI20} shows that such an output dissipation rate is not enough to guarantee uniform convergence of the output. Nevertheless, it has been shown in \cite{KACH20} that, if the output dissipation rate happens to be non-increasing along the system's solution, then the sought uniformity does hold.

However, this result is only available in the absence of inputs. The objective of this note is to extend it to systems with inputs, in order to guarantee both uniform robust output stability and Input-to-Output Stability (IOS), as studied in \cite{SONWANIOS,SONWANIOS-LYA,KAJIbook11}. In the first case, we require that disturbances have no qualitative impact on the output transients and steady-state behavior. In the second case, similar to Input-to-State Stability (ISS, \cite{cetraro,MironchenkoBook}), IOS imposes a uniform convergence to the origin, up to a steady-state error that is ``proportional'' to the input magnitude, with the crucial difference that IOS imposes this only on the considered output.

In the spirit of \cite{KACH20}, we show in Section \ref{sec-URGAOS} that, if the Lyapunov function admits a dissipation rate that involves the output only and if this dissipation rate happens to decrease along the system's solutions, then uniform robust global asymptotic output stability is guaranteed. The conditions we provide here are actually less restrictive than \cite{KACH20} even in the disturbance-free case. Similarly, in Section \ref{sec-IOS} we show that if the considered Lyapunov function admits an output dissipation, then the IOS property holds provided that this dissipation rate does not increase along solutions when the output is large compared to the input. 
We provide an academic example that shows the advantage of the method compared to existing Lyapunov conditions for IOS (which requires a dissipation rate involving the whole Lyapunov function). 

In the case where no uniformity requirement is imposed, \cite{KACH20} provided a Barb\u{a}lat condition to ensure output asymptotic stability. More precisely, it is shown there that the uniform continuity assumption in Barb\u{a}lat's lemma can be relaxed to a weaker property, thus leading to less demanding requirements for output stability in the absence of inputs. We show in Section \ref{sec-Barbalat} that a similar result can be derived for systems with disturbances to ensure robust global asymptotic output stability.

Finally, in Section \ref{sec-adaptive}, we show how our findings prove useful in the specific context of adaptive control. To that end, we rely on the adaptive control strategy recently developed in the book \cite{KAKRbook25}. This control strategy, called Deadzone-Adapted Disturbance Suppression (DADS), guarantees rejection of the steady-state effect of disturbances using deadzone and nonlinear damping. On a specific illustrative system, we show that our results enlarge the parameter range in which uniform asymptotic output stability is guaranteed, and we demonstrate the IOS property for a set of parameter values.

\vspace{3mm}
\textbf{Notation.} Given a set $D\subseteq \mR^m$, $L^\infty(\mRp,D)$ (resp. $L^\infty_{loc}(\mRp,D)$) denotes the set of all measurable and essentially bounded (resp. locally essentially bounded) functions $d:\mRp\to D$. We use $|\cdot|$ to denote the Euclidean norm of a vector and $\|\cdot\|_\infty$ for the $L^\infty$-norm of a signal (given $d\in L^\infty_{loc}(\mRp,D)$, $\|d\|_\infty=\textrm{ess\,sup}_{t\geq 0}|d(t)|$). The class $\mathcal{K}$ denotes the set of all continuous increasing functions $\alpha:\mRp\to\mRp$ satisfying $\alpha(0)=0$. $\alpha\in\cK_\infty$ if $\alpha\in\cK$ and $\lim_{s\to+\infty}\alpha(s)=+\infty$. A function $\beta:\mRp\to\mRp$ is of class $\cKL$ if, $\beta(\cdot,t)\in\cK$ for each fixed $t\geq 0$ and, for each fixed $s\geq 0$, $\beta(s,\cdot)$ is continuous non-decreasing and tends to zero as its argument tends to infinity. The gradient of a continuously differentiable function $V:\mR^n\to\mR$ is denoted by $\nabla V$. The rectification function is denoted by $(\cdot)^+$:  for any $s\in\mR$, $\left(s\right)^{+} :=\max \left\{0,s\right\}$. 

\section{Robust stability}\label{sec2}

\noindent Let $D\subseteq {\mathbb R}^{m} $ be a non-empty set and $f:{\mathbb R}^{n} \times D\to {\mathbb R}^{n} $ be a continuous mapping which is locally Lipschitz in its first variable, uniformly in the second variable on bounded set of $D$
, in the sense that, for every compact set $S\subseteq {\mathbb R}^{n}$ and for every bounded set $K\subseteq D$ there exists $L_{S,K}>0$ such that, for all $x,z\in S$ and all $d \in K$,
\[
\left|f(x,d)-f(z,d)\right|\le L_{S,K}\left|x-z\right|. 
\]
In this section, we assume that $f(0,d)=0$ for all $d\in D$, meaning that the origin is an equilibrium regardless of the applied disturbance $d$. Given a continuous mapping $h:{\mathbb R}^{n} \to {\mathbb R}^{k} $ with $h(0)=0$, consider the system
\begin{equation} \label{GrindEQ__1_} 
\dot{x}=f(x,d) 
\end{equation} 
with input $d\in L_{loc}^{\infty } \left(\mRp ;D\right)$ and output 
\begin{equation} \label{GrindEQ__2_} 
y=h(x).
\end{equation} 
Throughout this paper, we assume that system \eqref{GrindEQ__1_} is forward complete, i.e., for every $x_{0} \in {\mathbb R}^{n} $ and $d\in L_{loc}^{\infty } \left(\mRp ;D\right)$, the unique solution $x(t)=\phi (t,x_{0} ;d)$ of the initial-value problem \eqref{GrindEQ__1_} with initial condition $x(0)=x_{0}$ and input $d$ exists for all $t\ge 0$. We also use the notation $y(t,x_{0} ;d):=h(\phi (t,x_{0} ;d))$ for all $t\ge 0$.

The output defined in \eqref{GrindEQ__2_} should be seen as a subset of state variables, or a given quantity of interest, for which we want to impose a specific behavior regardless of the evolution of the other state variable. We recall below different output stability notions that capture this.

\begin{definition}[Output stability/attractivity] \label{def-1} 
We say that system \eqref{GrindEQ__1_}-\eqref{GrindEQ__2_}  is
\begin{itemize}
\item\textit{Robustly Globally Output Attractive (RGOA)} if, for every $x_{0} \in {\mathbb R}^{n} $  and every $d\in L_{loc}^{\infty } \left(\mRp ;D\right)$, it holds that $\lim_{t\to +\infty } y(t,x_{0} ;d)=0$.

\item \textit{Uniformly Robustly Globally Output Attractive (URGOA)} if, for every $\varepsilon ,R>0$, there exists $T(\varepsilon ,R)>0$ such that, for all $d\in L_{loc}^{\infty } \left(\mRp ;D\right)$ and all $x_{0} \in {\mathbb R}^{n}$ with $\left|x_{0} \right|\le R$, it holds that $\left|y(t,x_{0} ;d)\right|\le \varepsilon $ for all $t\ge T(\varepsilon ,R)$.

\item \textit{Robustly Lyapunov output stable} if, for every $\varepsilon >0$, there exists $\delta (\varepsilon )>0$ such that, for all $d\in L_{loc}^{\infty } \left(\mRp ;D\right)$ and all $x_{0} \in {\mathbb R}^{n} $ with $\left|x_{0} \right|\le \delta \left(\varepsilon \right)$, it holds that $\left|y(t,x_{0} ;d)\right|\le \varepsilon $ for all $t\ge 0$.

\item \textit{Robustly Lagrange output stable} if, for every $\delta >0$, there exists $\varepsilon(\delta) >0$ such that, for all $d\in L_{loc}^{\infty } \left(\mRp ;D\right)$ and all $x_{0} \in {\mathbb R}^{n} $ with $\left|x_{0} \right|\le \delta $, it holds that $\left|y(t,x_{0} ;d)\right|\le \varepsilon(\delta)$ for all $t\ge 0$.

\item \textit{Robustly Globally Asymptotically Output Stable (RGAOS)} if it is robustly Lagrange output stable, robustly Lyapunov output stable and RGOA.

\item \textit{Uniformly Robustly Globally Asymptotically Output Stable (URGAOS)} if it is robustly Lagrange output stable, robustly Lyapunov output stable and URGOA.
\end{itemize}
If no inputs are present, i.e., if $D$ is a singleton, then all above notions are used without the word ``robustly''. 
\end{definition}


RGOA simply means that any solution generates an output that asymptotically vanishes. URGOA additionally imposes that the convergence rate of the output to zero is uniform over bounded sets of initial states (and over all considered inputs). This uniformity precludes the possibility to have arbitrarily slow output transients when initial states are constrained to a prescribed bounded set, which constitutes an interesting feature in practice. Lyapunov and Lagrange stability are used in their classical sense, with the notable peculiarity that they concern the output only (and not necessarily the full state) and that the bounds on the output are uniform over all possible inputs. Finally, both RGAOS and URGAOS combine Lyapunov and Lagrange output stability and the convergence of the output to zero, with the additional requirement for URGAOS that this convergence is uniform over bounded sets of initial states.

Using \cite[Lemma 15]{albertini1999continuous} or Lemma 2.2, p.\,60, and Lemma 2.6, p.\,67, in \cite{KAJIbook11}, URGAOS can be equivalently stated as a $\cKL$ estimate (the reader should notice that Lemmas 2.2 and 2.6 in \cite{KAJIbook11} make no use of robust forward completeness, which is not assumed in Definition \ref{def-1}). 

\begin{proposition}[$\cKL$ estimate for URGAOS] The system \eqref{GrindEQ__1_}-\eqref{GrindEQ__2_} is URGAOS if and only if there exists a function $\beta \in \cKL$ such that, for all $x_{0} \in {\mathbb R}^{n} $ and all $d\in\cL(\mRp,D)$,
\begin{equation} \label{KL1} 
\left|y(t,x_{0} ;d)\right|\le \beta \left(\left|x_{0} \right|,t\right),\quad \forall t\geq 0.
\end{equation} 
\end{proposition}

\comment{
Some of the results presented here slightly simplify when additional regularity (in the input) is assumed on the vector field, namely when the following assumption holds.
}

Some of the results presented here can be slightly simplified under additional regularity assumptions on $f$ with respect to the input.  For results where this assumption is required, it will be explicitly indicated.

\begin{assumption}[Extra condition on the vector field]\label{ass-1} The set $D\subseteq {\mathbb R}^{m} $ is compact and the mapping $f:{\mathbb R}^{n} \times D\to {\mathbb R}^{n} $ is locally Lipschitz, i.e., for every compact set $S\subseteq {\mathbb R}^{n} $ there exists $L_S>0$ such that, for all $x,z\in S$ and all $d, e \in D$,
\[\left|f(x,d)-f(z,e)\right|\le L_S\left(\left|x-z\right| + \left|d-e\right|\right). \] 
\end{assumption}

Under this assumption, \cite{ANGSON-UFC} guarantees that forward completeness of \eqref{GrindEQ__1_}-\eqref{GrindEQ__2_} ensures robust forward completeness, in the sense that, starting from any bounded set of initial conditions and for arbitrary inputs, solutions can only explore a bounded region of the state space over any bounded time interval. In this case, by using results from \cite{KAJIbook11}  (namely Lemma 2.1, p.\,58, and Lemma 2.2, p.\,60), it can be seen that URGAOS boils down to URGOA (both Lyapunov and Lagrange output stability come for free).

\begin{proposition}[When URGOA $\Rightarrow$ URGOAS, \cite{KAJIbook11}]\label{prop-2} Under Assumption \ref{ass-1}, system \eqref{GrindEQ__1_}-\eqref{GrindEQ__2_} is URGAOS if and only if it is URGOA.
\end{proposition}

It should be noted that most of the results that are given below do not require the validity of Assumption \ref{ass-1}.

A way to guarantee robust Lagrange output stability is through the following result.

\begin{proposition}[Condition for Lagrange output stability]\label{prop-Lagrange} If there exist functions $Q\in C^{1} \left({\mathbb R}^{n} ;\mRp \right)$ and $a\in \cK_{\infty }$ such that, for all $x\in {\mathbb R}^{n}$ and all $d\in D$,
\begin{align} 
a\left(\left|h(x)\right|\right)\le Q(x) \label{GrindEQ__4_} \\
\nabla Q(x)f(x,d)\le 0,\label{GrindEQ__3_} 
\end{align} 
then the system \eqref{GrindEQ__1_}-\eqref{GrindEQ__2_} is robustly Lagrange output stable. 
\end{proposition}

Note that these conditions are imposed to hold for all $x$, but $Q(x)$ is not requested to vanish for $x=0$. Similar conditions imposed locally and with a function that vanishes at zero lead to robust Lyapunov output stability.

\begin{proposition}[Condition for Lyapunov output stability]\label{prop3} If there exist a constant $r>0$, functions $W\in C^{1} \left({\mathbb R}^{n} ;\mRp \right)$ with $W(0)=0$ and $a\in \cK$ such that, for all $x\in {\mathbb R}^{n} $ with $W(x)<r$ and all $d\in D$,
\begin{align} 
a\left(\left|h(x)\right|\right)\le W(x) \label{GrindEQ__5_} \\
\nabla W(x)f(x,d)\le 0,\label{GrindEQ__6_} 
\end{align} 
then \eqref{GrindEQ__1_}-\eqref{GrindEQ__2_} is robustly Lyapunov output stable. 
\end{proposition}

\subsection{Robust UGAOS}\label{sec-URGAOS}

As shown by the counter-example in \cite{ORCHSI20}, uniform output stability notions cannot be ensured in general by a Lyapunov function that dissipates in terms of the output only. Yet, it has been shown in \cite{KACH20} that URGOAS does hold if, in addition, this dissipation does not increase along the system's solutions. The following statement relaxes this assumption by imposing this extra condition only when the dissipation term is small and generalizes that result by allowing systems with inputs.

\begin{theorem}[Conditions for URGOA and URGAOS]\label{thm_1}Suppose that there exist $V,W\in C^{1} \left({\mathbb R}^{n} ;\mRp \right)$ and a continuous positive definite function $\rho :\mRp \to \mRp $ such that, for all $x\in {\mathbb R}^{n} $ and all $d\in D$,
\begin{align} 
\nabla V(x)f(x,d)\le -\rho \left(W(x)\right). \label{GrindEQ__7_} 
\end{align}
Assume further that, for some constant $r>0$ and some function $a\in \cK$, it holds for all $x\in {\mathbb R}^{n}$ with $W(x)<r$ and all $d\in D$ that
\begin{align}
a\left(\left|h(x)\right|\right)\le W(x)  \label{GrindEQ__5_bis} \\
\nabla W(x)f(x,d)\le 0. \label{GrindEQ__6_bis} 
\end{align}
Then, under any of the two following conditions:
\begin{itemize}
    \item [$i)$] $\liminf_{s\to +\infty } \,\rho (s)>0$,
    \item[$ii)$] there exist $\zeta \in \cK_{\infty } $ and $Q\in C^{1} \left({\mathbb R}^{n} ;\mRp \right)$ such that, for all $x\in {\mathbb R}^{n} $ and all $d\in D$,
\begin{align} \label{GrindEQ__8_} 
W(x)\le \zeta \left(Q(x)\right)\\
\nabla Q(x)f(x,d)\le 0,\label{GrindEQ__3_bis} 
\end{align} 
\end{itemize}
the system \eqref{GrindEQ__1_}-\eqref{GrindEQ__2_}  is URGOA and robustly Lyapunov stable. In particular, if Assumption \ref{ass-1} holds, then the system \eqref{GrindEQ__1_}-\eqref{GrindEQ__2_} is URGAOS.
\end{theorem}

It is worth noting that \eqref{GrindEQ__5_bis} imposes that, locally, $W$ does not vanish unless the output is zero. Thus, the dissipation term $\rho(W(x))$ in \eqref{GrindEQ__7_} can very well involve only output terms. In other words, \eqref{GrindEQ__7_} allows for a Lyapunov function $V$ that dissipates in terms of the output only. Condition \eqref{GrindEQ__6_bis} then requests that this dissipation term does not increase along solutions when it is small. URGOA can then be concluded provided that either the function $\rho$ does not vanish at infinity (condition $i)$) or that $W$ is dominated by a function $Q$ that does not increase along solutions (condition $ii)$). Theorem 1 in \cite{KACH20} is thus a particular case of the above result, for which $Q=W$ and $\zeta(s)=s$ for all $s\geq 0$.

\vspace{3mm}
\begin{proof}[of Theorem \ref{thm_1}]
Notice that \eqref{GrindEQ__7_} and the fact that $f(0,d)=0$ for all $d\in D$ imply that $W(0)=0$. Robust Lyapunov output stability then follows from Proposition \ref{prop3}. We now proceed to establishing URGOA. By \eqref{GrindEQ__7_} and the assumed forward completeness, it holds for all $x_{0} \in {\mathbb R}^{n}$, all $d\in L_{loc}^{\infty} \left(\mRp ;D\right)$ and almost all $t\geq 0$ that
\begin{equation} \label{GrindEQ__9_} 
\frac{d}{d\, t} V\left(\phi (t,x_{0} ;d)\right)\le -\rho \left(W\left(\phi(t,x_{0} ;d)\right)\right).
\end{equation} 
We next establish the following.
\begin{fact}\label{fact-1}
    Given any $x_{0} \in {\mathbb R}^{n} $, any $d\in L_{loc}^{\infty } \left(\mRp ;D\right)$ and any $t\ge 0$, one of the following holds:
    \begin{itemize}
        \item $W(\phi (\tau ,x_{0} ;d))\ge r$ $\,\forall \tau \in \left[0,t\right]$, and $W\left(\phi (t,x_{0} ;d)\right)>r$
        \item $W(\phi (t,x_{0} ;d))\le \min \left\{r,W\left(\phi (\tau ,x_{0} ;d)\right)\right\}$ $\,\forall \tau \in \left[0,t\right]$.
    \end{itemize}
\end{fact}

\begin{proof}
Consider the continuously differentiable function
\begin{equation} \label{GrindEQ__11_} 
U(x):=\frac{1}{2} \left(\left(r-W(x)\right)^{+} \right)^{2},\quad \forall x\in\mR^n.
\end{equation} 
Then it holds that 
\begin{align*}
\nabla U(x)f(x,d)=-\left(r-W(x)\right)^{+} \nabla W(x)f(x,d).
\end{align*}
By \eqref{GrindEQ__6_}, it follows that, for all $x\in {\mathbb R}^{n} $ and all $d\in D$,
\begin{equation} \label{GrindEQ__12_} 
\nabla U(x)f(x,d)\ge 0. 
\end{equation} 
Hence, for every $x_{0} \in {\mathbb R}^{n} $ and any $d\in L_{loc}^{\infty } \left(\mRp ;D\right)$, the mapping $t\mapsto U\left(\phi (t,x_{0} ;d)\right)$ is non-decreasing and we get from \eqref{GrindEQ__11_} that, for all $t\ge 0$ and all $\tau \in \left[0,t\right]$,
\begin{equation*} 
\left(r-W\left(\phi (t,x_{0} ;d)\right)\right)^{+} \ge \left(r-W\left(\phi (\tau ,x_{0} ;d)\right)\right)^{+}.  
\end{equation*} 
Fact \ref{fact-1} follows by considering the cases $W\left(\phi (t,x_{0} ;d)\right)>r$ and $W\left(\phi (t,x_{0} ;d)\right)\le r$. 
\end{proof}

\vspace{3mm}
The proof of Theorem \ref{thm_1} continues along two different lines according to whether $i)$ or $ii)$ holds.

\noindent \underline{Case $i)$}: $\liminf_{s\to+\infty} \rho(s)>0$. Let arbitrary $\varepsilon ,R>0$ be given and define the two following quantities:
\begin{align} 
\tilde{\rho }(\varepsilon )&:=\inf \left\{\rho (s)\, :\, \, s\ge \min \left\{r,a(\varepsilon )\right\}\, \right\}  \label{GrindEQ__14_} \\
T(\varepsilon ,R)&:=\frac{1+\max \left\{\, V(x)\, :\,  \left|x\right|\le R\, \right\}}{\tilde{\rho }(\varepsilon )}.  \label{GrindEQ__15_} 
\end{align} 
Since $\rho$ is positive definite, $i)$ ensures that $\tilde{\rho }(\varepsilon )>0$. We next prove by contradiction that $\left|y(t,x_{0} ;d)\right|\le \varepsilon $ for all $t\ge T(\varepsilon ,R)$, all $d\in L_{loc}^{\infty } \left(\mRp ;D\right)$ and all $x_{0} \in {\mathbb R}^{n} $ with $\left|x_{0} \right|\le R$, thus establishing URGOA. To that end, suppose that there exists $d\in L_{loc}^{\infty } \left(\mRp ;D\right)$, $x_{0} \in {\mathbb R}^{n} $ with $\left|x_{0} \right|\le R$ and $t\ge T(\varepsilon ,R)$ such that $\left|y(t,x_{0} ;d)\right|>\varepsilon $. By virtue of \eqref{GrindEQ__5_} and by distinguishing the cases $W\left(\phi (t,x_{0} ;d)\right)>r$ and $W\left(\phi (t,x_{0} ;d)\right)\le r$ we conclude that $W\left(\phi (t,x_{0} ;d)\right)\ge \min \left\{r,a(\varepsilon )\right\}$. It follows from Fact \ref{fact-1} that 
either $W\left(\phi (\tau ,x_{0} ;d)\right)\ge r$ for all $\tau \in \left[0,t\right]$, or $\min \left\{r,a(\varepsilon )\right\}\le W\left(\phi (t,x_{0} ;d)\right)\le W\left(\phi (\tau ,x_{0} ;d)\right)$ for all $\tau \in \left[0,t\right]$. In view of \eqref{GrindEQ__14_}, we get
\begin{equation} \label{GrindEQ__17_}
\tilde{\rho }(\varepsilon )\le \rho \left(W\left(\phi (\tau,x_{0} ;d)\right)\right), \quad \forall \tau\in \left[0,t\right].
\end{equation}                 
On the other hand, \eqref{GrindEQ__9_} ensures that
\begin{equation} \label{GrindEQ__18_} 
V\left(\phi (t,x_{0} ;d)\right)\le V\left(x_{0} \right)-\int _{0}^{t}\rho \left(W\left(\phi (\tau,x_{0} ;d)\right)\right)d\tau. 
\end{equation} 
Combining \eqref{GrindEQ__17_} and \eqref{GrindEQ__18_}, we get the inequality
\begin{equation} \label{GrindEQ__19_} 
V\left(\phi (t,x_{0} ;d)\right)\le V\left(x_{0} \right)-\tilde{\rho }(\varepsilon)t. 
\end{equation} 
Since $\left|x_{0} \right|\le R$, we get from \eqref{GrindEQ__19_} that
\begin{equation*} 
V\left(\phi (t,x_{0} ;d)\right)\le \max \left\{V(x)\, :\, \left|x\right|\le R\, \right\}-\tilde{\rho }(\varepsilon )t,
\end{equation*} 
Using the fact that $t\ge T(\varepsilon ,R)$ with $T$ defined in \eqref{GrindEQ__15_}, we get the contradiction that $V\left(\phi (t,x_{0} ;d)\right)<0$.

\vspace{2mm}
\noindent \underbar{Case $ii)$:} there exist $\zeta \in \cK_\infty$ and $Q\in C^{1} \left({\mathbb R}^{n} ;\mRp \right)$ satisfying \eqref{GrindEQ__8_}-\eqref{GrindEQ__3_bis}. In this case, it holds for all $d\in L_{loc}^{\infty } \left(\mRp ;D\right)$ and all $x_{0} \in {\mathbb R}^{n} $ that
\begin{equation*} 
Q\left(\phi (t,x_{0} ;d)\right)\le Q\left(x_{0} \right),\quad \forall t\geq 0. 
\end{equation*} 
In conjunction with \eqref{GrindEQ__8_}, this implies that
\begin{equation} \label{GrindEQ__22_} 
W\left(\phi (t,x_{0} ;d)\right)\le \zeta \left(Q\left(x_{0} \right)\right),\quad \forall t\geq 0.
\end{equation} 
Let arbitrary $\varepsilon ,R>0$ be given and define
\begin{align*} 
\bar{\rho }(\varepsilon ,R)&:=\min \left\{\rho (s)\, :\, s\in[\min \left\{r,a(\varepsilon )\right\}, r+\bar\zeta(R)]\right\}\\ 
\bar{T}(\varepsilon ,R)&:=\frac{1+\max \left\{\, V(x)\, :\, \left|x\right|\le R\, \right\}}{\bar{\rho }(\varepsilon ,R)}.
\end{align*} 
where $\bar\zeta(R):=\zeta \left(\max \left\{\, Q(x)\, :\, \left|x\right|\le R\right\}\right)$.
Here again we prove by contradiction that $\left|y(t,x_{0} ;d)\right|\le \varepsilon $ for all $t\ge \bar{T}(\varepsilon ,R)$, all $d\in L_{loc}^{\infty } \left(\mRp ;D\right)$ and all $x_{0} \in {\mathbb R}^{n} $ with $\left|x_{0} \right|\le R$, thus establishing URGOA. To that end, suppose on the contrary that there exists $d\in L_{loc}^{\infty } \left(\mRp ;D\right)$, $x_{0} \in {\mathbb R}^{n} $ with $\left|x_{0} \right|\le R$ and $t\ge \bar{T}(\varepsilon ,R)$ such that $\left|y(t,x_{0} ;d)\right|>\varepsilon $. By virtue of \eqref{GrindEQ__5_} and by distinguishing the cases $W\left(\phi (t,x_{0} ;d)\right)>r$ and $W\left(\phi (t,x_{0} ;d)\right)\le r$ we get that $W\left(\phi (t,x_{0} ;d)\right)\ge \min \left\{r,a(\varepsilon )\right\}$. It follows from Fact \ref{fact-1} and \eqref{GrindEQ__22_} that either $\zeta \left(Q\left(x_{0} \right)\right)\ge W\left(\phi (\tau ,x_{0} ;d)\right)\ge r$ for all $\tau \in \left[0,t\right]$, or $\min \left(r,a(\varepsilon )\right)\le W\left(\phi (t,x_{0} ;d)\right)\le W\left(\phi (\tau ,x_{0} ;d)\right)\le \zeta \left(Q\left(x_{0} \right)\right)$ for all $\tau \in \left[0,t\right]$. Since $\left|x_{0} \right|\le R$, we get $Q\left(x_{0} \right)\le \max \left\{Q(x)\, :\, \left|x\right|\le R\, \right\}$ and it follows that
\begin{equation} \label{GrindEQ__26_}
\bar{\rho }(\varepsilon ,R)\le \rho \left(W\left(\phi (\tau,x_{0} ;d)\right)\right), \quad \forall \tau\in \left[0,t\right].
\end{equation} 
Continuing exactly as in Case $i)$, we arrive at the contradiction that $V\left(\phi (t,x_{0} ;d)\right)<0$.

Finally, if Assumption \ref{ass-1} holds, then URGAOS follows from URGAO by Proposition \ref{prop-2}.
\end{proof}

\subsection{Non-uniform notions}\label{sec-Barbalat}

An alternative way to ensure RGOA is with Barb\u{a}lat's lemma, which states that any integrable function tends to zero, provided that it is uniformly continuous (see for instance \cite{KHALIL2002}). In \cite{KACH20}, this classical result has been enhanced, by relaxing uniform continuity of the considered signal to the following.

\begin{definition}[Quasi-uniform continuity] A function $\varphi:\mRp \to {\mathbb R}$ is said to be quasi-uniformly continuous if, for every $\varepsilon >0$, there exists $\delta (\varepsilon )>0$ such that $\varphi(t)-\varphi(t_{0} )<\varepsilon $ for all $t_{0} \ge 0$ and all $t\in[t_0,t_0+\delta (\varepsilon )]$. 
\end{definition}

A sufficient condition to guarantee quasi-uniform continuity has been provided in \cite[Proposition 2]{KACH20}, which can be slightly relaxed as follows.

\begin{proposition}[On quasi-uniform continuity]\label{prop-5} The function $\varphi:\mRp \to {\mathbb R}$ is quasi-uniformly continuous if at least one of the following statements holds:
\begin{itemize}
    \item[$i)$] there exists $M>0$ such that the function $\psi:t\mapsto \varphi(t)-Mt$ is non-increasing on $\mRp$;
    \item[$ii)$] $\varphi$ is continuous and there exists $M\ge 0$ such that $\limsup_{h\to 0^{+} } \frac{\varphi(t+h)-\varphi(t)}{h}\le M$ for all $t\ge 0$;
    \item[$iii)$] $\varphi$ is absolutely continuous and there exists $M\ge 0$ such that $\dot{\varphi}(t)\le M$ for almost all $t\ge 0$.
\end{itemize}
\end{proposition}


We also report the following result.

\begin{proposition}[Saturated functions]\label{prop-6}
    If the function $\varphi:\mRp\to\mR$ is absolutely continuous and there exist constants $M\ge 0$ and  $N\in {\mathbb R}$ such that $\dot{\varphi}(t)\le M$ for almost all $t\in \left\{\, \tau\ge 0\, :\, \varphi(\tau)<N\right\}$, then the function $\tilde \varphi:\mRp\to\mR$ defined as $\tilde \varphi(t):=\min\{\varphi(t),N\}$ is quasi-uniformly continuous.
\end{proposition}

\begin{proof}
Consider the function $\psi:\mRp\to\mR$ defined as 
\begin{align*}
    \psi(t):=\min \{\varphi(t),N\}-Mt,\quad \forall t\geq 0.
\end{align*}
We show next that
\begin{align}\label{eq-1}
    \limsup_{h\to 0^{+} } \frac{\psi(t+h)-\psi(t)}{h} \le 0,\quad \forall t\ge 0.
\end{align}
By virtue of the results in \cite[Chapter 6]{BARO01}, \eqref{eq-1} guarantees that $\psi$ is non-increasing and Proposition \ref{prop-6} then follows from Proposition \ref{prop-5}. In order to establish \eqref{eq-1}, we consider two cases. In the case when $\varphi(t)<N$ then (by continuity) there exists $\delta >0$ for which $\varphi(t+h)<N$ for all $h\in \left(0,\delta \right)$. By assumption, this guarantees that $\dot{\varphi}(\tau)\le M$ for almost all $\tau\in \left(t,t+\delta \right)$. Consequently, by absolute continuity of $\varphi$, we get that, for all $h\in \left(0,\delta \right)$,
\begin{align*}
\min \{\varphi(t+h),N\}&=\varphi(t+h)\\
&=\varphi(t)+\int _{t}^{t+h}\dot{\varphi}(\tau)d\tau \\
&\le \varphi(t)+Mh=\min \{\varphi(t),N\}+Mh.
\end{align*}
Therefore, we obtain \eqref{eq-1}. On the other hand, if $\varphi(t)\ge N$ then for all $h\geq 0$ we have that
\begin{align*}
\psi(t)&=N-Mt\\
&\ge \min \{\varphi(t+h),N\}-Mt\\
&\ge \min \{\varphi(t+h),N\}-M (t+h)=\psi(t+h),
\end{align*}
and \eqref{eq-1} follows in that case too.    
\end{proof}

One of the key contributions of \cite{KACH20} was to show that Barb\u{a}lat's lemma, originally stated for uniformly continuous functions, remains valid for quasi-uniformly continuous functions, as recalled next.

\begin{lemma}[Extension of Barb\u{a}lat's lemma, \cite{KACH20}]\label{lem_barbalat} Let $\rho :\mRp \to \mRp $ be a continuous, positive definite function and let $\varphi :\mRp \to \mRp $ be a given function such that either $\varphi $ or $-\varphi $ is quasi-uniformly continuous. Suppose that either $\rho$ is non-decreasing or $\varphi $ is bounded. Then the following implication holds:
\begin{align*}
    \int_0^{+\infty} \rho(\varphi(t))dt<+\infty\quad \Rightarrow \quad \lim_{t\to +\infty }\varphi (t)=0.
\end{align*}
\end{lemma}
This lemma was used in \cite[Theorem 2]{KACH20} to provide a Lyapunov condition for asymptotic output stability of input-free systems. The following result extends that result by allowing for systems with inputs.

\begin{theorem}[Lyapunov condition for RGOA]\label{thm-2} Suppose that there exist a constant $r>0$, functions $V,W,Q\in C^{1} \left({\mathbb R}^{n} ;\mRp \right)$, $a\in \cK$ and a continuous, positive definite function $\rho :\mRp \to \mRp $ such that, for all $x\in\mR^n$ and all $d\in D$,
\begin{align}
    \nabla V(x)f(x,d)&\leq -\rho(W(x)) \label{GrindEQ__7_ter} \\
    \nabla Q(x)f(x,d)&\leq 0, \label{GrindEQ__3_ter} 
\end{align}
and, for all $x\in\mR^n$ with $W(x)<r$,
\begin{align}\label{GrindEQ__5_ter}
    a(|h(x)|)\leq W(x).
\end{align}
Moreover, suppose that there exists a continuous function $\gamma :\mRp \to \mRp $ such that one of the following holds:
\begin{itemize}
    \item[$i)$] for all $d\in D$ and for all $x\in {\mathbb R}^{n} $ with $W(x)<r$:
    \begin{equation} \label{GrindEQ__27_} 
\nabla W(x)f(x,d)\le \gamma \left(Q(x)\right); 
\end{equation} 
\item[$ii)$]  for all $x\in {\mathbb R}^{n} $ and all $d\in D$: 
\begin{equation} \label{GrindEQ__28_} 
\nabla W(x)f(x,d)\ge -\gamma \left(Q(x)\right). 
\end{equation} 
\end{itemize}
Finally, suppose that either $\liminf_{s\to +\infty } \rho (s)>0$ or that there exists a function $\zeta \in \cK_{\infty }$ such that
\begin{align} \label{GrindEQ__8_ter} 
W(x)\le \zeta \left(Q(x)\right),\quad \forall x\in\mR^n.
\end{align} 
Then the system \eqref{GrindEQ__1_}-\eqref{GrindEQ__2_}  is RGOA.
\end{theorem}

In the case when $\rho$ does not vanish at infinity, this theorem states that RGOA can be established based on a dissipation rate that involves a term $W(x)$ (possibly depending on the output only) provided that this function $W$ either increases at most at an input-independent rate (condition \eqref{GrindEQ__27_}) or decreases not faster than an input-free rate (condition \eqref{GrindEQ__28_}). Hence, the conditions imposed on $W$ are much less stringent than in Theorem \ref{thm_1}, at the price of a possible lack of uniformity of the obtained output stability property.

Even in the output-free case, Theorem \ref{thm-2} constitutes a novelty with respect to \cite[Theorem 2]{KACH20} where the function $Q$ was taken as $V$ itself and condition $i)$ was imposed for all $x\in\mR^n$.

We also stress that $i)$ is imposed only for small values of $W(x)$ whereas $ii)$ is imposed for all $x\in\mR^n$. This discrepancy results from the proof techniques we employ here. We do not know whether the result would still be valid if $ii)$ was imposed only for small values of $W(x)$.

\begin{proof}[of Theorem \ref{thm-2}] Let arbitrary $x_{0} \in {\mathbb R}^{n} $ and $d\in L_{loc}^{\infty } \left(\mRp ;D\right)$ be given. Notice that \eqref{GrindEQ__7_} and the forward completeness assumption imply that, for all $t\geq 0$,
\begin{equation} \label{GrindEQ__29_} 
\hspace{-8pt} V\left(\phi (t,x_{0} ;d)\right)+\int _{0}^{t}\rho \left(W\left(\phi (\tau,x_{0} ;d)\right)\right)d\tau \le V\left(x_{0} \right).
\end{equation} 
Moreover, \eqref{GrindEQ__3_ter} guarantees  that  
\begin{align}\label{GrindEQ__21_ter}
    Q(\phi(t,x_0;d))\leq Q(x_0),\quad \forall t\geq 0
\end{align}
We divide the proof in four cases, according to which assumption in the theorem is fulfilled.

\noindent \underline{Case 1:} condition $i)$ holds and $\liminf_{s\to +\infty } \rho (s)>0$. In this case, we can define the non-decreasing function:
\begin{align}\label{GrindEQ__30_} 
    \tilde \rho(s):=\left\{\begin{array}{cl}
    \inf_{z\geq s} \rho(z)&\quad \textrm{if } s\geq 0\\
    0&\quad \textrm{if } s<0.
    \end{array}\right.
\end{align}
We also define
\begin{equation} \label{GrindEQ__31_} 
\bar{\rho }\left(s\right):=\int _{-1}^{0}\tilde{\rho }\left(s+z\right)dz,\quad \forall s\geq 0. 
\end{equation} 
This function $\bar{\rho }:\mRp \to \mRp $ is then continuous, positive definite and non-decreasing and it satisfies $\bar{\rho }\left(s\right)\le \rho \left(s\right)$ for all $s\ge 0$. Consequently, we get from \eqref{GrindEQ__29_} that
\begin{equation} \label{GrindEQ__32_} 
\int _{0}^{t}\bar{\rho }\left(W\left(\phi (s,x_{0} ;d)\right)\right)ds \le V\left(x_{0} \right),\quad \forall t\geq 0.
\end{equation} 
Since $\bar{\rho }$ is non-decreasing, this ensures that
\begin{equation} \label{GrindEQ__33_} 
\int _{0}^{t}\bar{\rho }\left(\min \left\{W\left(\phi (\tau,x_{0} ;d)\right),r\right\}\right)d\tau \le V\left(x_{0} \right),
\end{equation} 
for all $t\ge 0$.
Since $W$ is continuously differentiable, the mapping $ t\mapsto W\left(\phi (t,x_{0} ;d)\right)$ is absolutely continuous on $\mRp$. Defining $M:=\max \left\{\, \gamma \left(s\right)\, :\, s\in[0,Q\left(x_{0} \right)]\, \right\}$, we get from \eqref{GrindEQ__27_} and \eqref{GrindEQ__21_ter} that, \textcolor{blue}{for almost all $t\geq 0$, 
\begin{equation} \label{GrindEQ__34_}
W\left(\phi (t,x_{0} ;d)\right)<r\ \Rightarrow \ \frac{d}{dt} \left(W\left(\phi (t,x_{0} ;d)\right)\right)\le M.
\end{equation} 
}
By Proposition \ref{prop-5}, $t\mapsto \min \left\{W\left(\phi (t,x_{0} ;d)\right),r\right\}$ is quasi-uniformly continuous and \eqref{GrindEQ__33_} ensures by Lemma \ref{lem_barbalat} that 
\[
\lim_{t\to +\infty } \min \left\{W\left(\phi (t,x_{0} ;d)\right),r\right\}=0.
\] 
Since $r>0$, we obtain $\lim_{t\to +\infty } W\left(\phi (t,x_{0} ;d)\right)=0$ which, by \eqref{GrindEQ__5_ter}, gives $\lim_{t\to +\infty } y(t,x_{0} ;d)=0$.

\noindent \underbar{Case 2:} condition $i)$ holds and there exists $\zeta \in \cK_{\infty }$  such that \eqref{GrindEQ__8_ter} holds. First observe that estimate \eqref{GrindEQ__21_ter} in conjunction with \eqref{GrindEQ__8_ter} implies that 
\begin{align}\label{GrindEQ__22_ter}
   W(\phi(t,x_0;d))\leq \zeta(Q(x_0)), \quad \forall t\geq 0.
\end{align}
\textcolor{blue}{Let $\underline \rho:\mRp\to\mRp$ be defined as 
\begin{align*}
\underline\rho(s):=\min \left\{\rho \left(z\right): z\hspace{-1mm}\in\hspace{-1mm} [\min \left\{s,\zeta \hspace{-1mm}\left(Q(x_{0} )\right)\right\}, \zeta\hspace{-1mm}\left(Q(x_{0} )\right)\hspace{-1mm}+\hspace{-1mm}1] \right\}
\end{align*}
for all $s\geq 0$ and consider} the non-decreasing function:
\begin{align}\label{GrindEQ__35_} 
     \tilde \rho(s):=\left\{\begin{array}{cl} \underline\rho(s)
    &\quad \textrm{if } s\geq 0\\
    0&\quad \textrm{if } s<0.
    \end{array}\right.
\end{align}
We also define $\bar{\rho }:\mRp \to \mRp $ by means of \eqref{GrindEQ__31_}. Then $\bar{\rho }$ is a continuous, positive definite and non-decreasing function that satisfies $\bar{\rho }\left(s\right)\le \rho \left(s\right)$ for all $s\in \left[0,\zeta \left(Q(x_{0} )\right)+1\right]$. Consequently, we get \eqref{GrindEQ__32_} from \eqref{GrindEQ__29_} and \eqref{GrindEQ__22_ter} \textcolor{red}. Continuing as in the proof of Case 1, we establish that $\lim_{t\to +\infty } y(t,x_{0} ;d)=0$.

\noindent \underbar{Case 3:} condition $ii)$ holds and $\liminf_{s\to +\infty } \rho (s)>0$. In this case we can define the function \textit{$\tilde{\rho }:{\mathbb R}\to \mRp $} by means of \eqref{GrindEQ__30_} and the function \textit{$\bar{\rho }:\mRp \to \mRp $} by means of \eqref{GrindEQ__31_}: the resulting function $\bar{\rho }$ is continuous, positive definite and non-decreasing and satisfies $\bar{\rho }\left(s\right)\le \rho \left(s\right)$ for all $s\ge 0$. Consequently, we get \eqref{GrindEQ__32_} from \eqref{GrindEQ__29_}. Since $W$ is continuously differentiable, the mapping $t\mapsto W\left(\phi (t,x_{0} ;d)\right)$ is absolutely continuous on $\mRp $. By virtue of \eqref{GrindEQ__28_} and \eqref{GrindEQ__21_ter},  it holds that
\begin{equation} \label{GrindEQ__36_} 
\frac{d}{d\, t} \left(-W\left(\phi (t,x_{0};d )\right)\right)\le M,\quad \forall t\geq 0\ a.e.,
\end{equation} 
with $M:=\max \left\{\gamma \left(s\right)\, :\, s\in[0,Q\left(x_{0} \right)]\right\}$. By Proposition \ref{prop3}, it follows that the mapping $t\mapsto -W\left(\phi (t,x_{0} ;d)\right)$ is quasi-uniformly continuous on $\mRp$. Therefore, \eqref{GrindEQ__32_} ensures by Lemma \ref{lem_barbalat} that $\lim_{t\to +\infty } W\left(\phi (t,x_{0} ;d)\right)=0$. By \eqref{GrindEQ__5_ter}, we conclude that $\lim_{t\to +\infty } y(t,x_{0} ;d)=0$. 

\noindent \underbar{Case 4:} condition $ii)$ holds and there exists $\zeta \in \cK_{\infty }$  such that  \eqref{GrindEQ__8_ter} holds. As in Case 2, the estimate \eqref{GrindEQ__21_ter} in conjunction with \eqref{GrindEQ__8_ter} implies the estimate \eqref{GrindEQ__22_ter}. Continuing as in the proof of Case 3, we establish (using Lemma \ref{lem_barbalat} and Proposition \ref{prop3}) that $\lim_{t\to +\infty } y(t,x_{0} ;d)=0$.
\end{proof}

\section{IOS based on output dissipation}\label{sec-IOS}

The results of Section \ref{sec2} guarantee asymptotic output stability for system \eqref{GrindEQ__1_}-\eqref{GrindEQ__2_} despite the presence of disturbances. In practice, these disturbances, especially when taking \textcolor{blue}{values in an unbounded set}, can significantly affect the output behavior, influencing both the transient overshoots and steady-state errors. The IOS and practical IOS notions are particularly well suited to characterize these effects.

Unlike in Section \ref{sec2}, we do not impose that $f(0,d)=0$ for all $d\in D$, but merely that the origin is an equilibrium when the input is zero:
\begin{align*}
    f(0,0)=0.
\end{align*}
Furthermore, for this section, we assume that $D\subseteq \mathbb R^m$ is closed with $0\in D$. We retain the same assumptions on $h$ and the forward completeness of \eqref{GrindEQ__1_}–\eqref{GrindEQ__2_} from Section \ref{sec2}.
The regularity property on $f$, however, slightly differs from that in Section \ref{sec2}. Specifically, we assume that $f: 
{\mathbb R}^{n} \times D\to {\mathbb R}^{n} $  is continuous with $f(0,0)=0$, and $f(\cdot, d)$ is locally Lipschitz 
uniformly in $d$ when $d$ takes values in any compact subset of $D$, i.e., for every pair of compact sets $S\subseteq {\mathbb R}^{n}$ and $K\subseteq D$ there exists $L_{S,K}>0$ such that, for all $x,z\in S$ and all $d \in K$,
\[
\left|f(x,d)-f(z,d)\right|\le L_{S,K}\left|x-z\right|. 
\]
We recall the notions of IOS and practical IOS \cite{SONWANIOS-LYA,SONWANIOS,KAJIbook11}.

\begin{definition}[p-IOS and IOS] The system \eqref{GrindEQ__1_}-\eqref{GrindEQ__2_} is said to be \emph{practically Input-to-Output Stable (p-IOS)} if there exist a function $\beta \in \cKL$ and a non-decreasing, continuous function $\gamma :\mRp \to \mRp $ such that, for all $x_{0} \in {\mathbb R}^{n} $ and all $d\in L^{\infty } \left(\mRp ;D\right)$, its output satisfies
\begin{equation} \label{GrindEQ__38_} 
\left|y(t,x_{0} ;d)\right|\le \beta \left(\left|x_{0} \right|,t\right)+\gamma \left(\left\| d\right\| _{\infty } \right),\quad \forall t\geq 0. 
\end{equation} 
It is said to be \emph{Input-to-Output Stable (IOS)} if, in addition, $\gamma (0)=0$. The function $\gamma$ is referred to a \emph{p-IOS (or IOS) gain} for \eqref{GrindEQ__1_}-\eqref{GrindEQ__2_}.
\end{definition}

\begin{remark}Due to causality, the estimate \eqref{GrindEQ__38_} is equivalent to requesting that
\[
\left|y(t,x_{0} ;d)\right|\le \beta \left(\left|x_{0} \right|,t\right)+\gamma \left(\textrm{ess\,sup}_{\tau\in[0,t]} \left|d(\tau)\right|\right),
\] 
for all $t\ge 0$,
all $x_{0} \in {\mathbb R}^{n} $, and all $d\in L^\infty_{loc}(\mRp,D)$. 
\end{remark}

As mentioned in the introduction, a Lyapunov function that dissipates in terms of the output norm is in general insufficient to establish IOS or p-IOS, as the $\cKL$ estimate in \eqref{GrindEQ__38_} imposes a uniform convergence of the output to the ball of radius $\gamma(\|d\|_\infty)$. The following result proposes an extra condition to ensure IOS or p-IOS based on a dissipation rate that involves only output terms. 




\begin{theorem}[Lyapunov condition for p-IOS] \label{theorem-lyap-ios1} 
Suppose that there exist continuously differentiable functions $V,W:\mR^n\to\mRp$ with $W(0)=0$, $a\in \cK_{\infty } $, a continuous non-decreasing function $\chi :\mRp \to \mRp $  and a continuous positive definite function $\rho :\mRp \to \mRp$ such that, for all $x\in {\mathbb R}^{n} $ and all $d \in D$, 
\begin{align} 
a\left(\left|h(x)\right|\right)&\le W(x) \label{GrindEQ__5_quatro} \\
W(x)\ge \chi \left(\left|d\right|\right)\ \Rightarrow\ \nabla V(x)&f(x,d)\le -\rho \left(W(x)\right).\label{GrindEQ__39_} 
\end{align} 
Assume further that, for all $x\in {\mathbb R}^{n} $ and all $d \in D$, 
\begin{equation} \label{GrindEQ__40_} 
W(x)\ge \chi \left(\left|d\right|\right)\quad\Rightarrow \quad \nabla W(x)f(x,d)\le 0.
\end{equation} 
Then the system \eqref{GrindEQ__1_}-\eqref{GrindEQ__2_}  is p-IOS with gain $a^{-1}\circ \chi$. In particular, if $\chi(0)=0$ then the system \eqref{GrindEQ__1_}-\eqref{GrindEQ__2_} is IOS.
\end{theorem}

Condition \eqref{GrindEQ__39_} imposes that $V$ dissipates in terms of $W(x)$ whenever the latter dominates the input norm. In view of \eqref{GrindEQ__5_quatro}, $W(x)$ can very well involve only output terms. The extra condition to conclude p-IOS or IOS is that this function $W$ does not increase along solutions, at least when it is large compared to the input norm.

Theorem \ref{theorem-lyap-ios1} generalizes the well-known Lyapunov condition for IOS \cite{SONWANIOS-LYA}, that requires \eqref{GrindEQ__5_quatro}-\eqref{GrindEQ__39_} to hold with $W=V$ (notice that, in such a case, \eqref{GrindEQ__40_} holds automatically). Theorem \ref{theorem-lyap-ios1} imposes less demanding conditions, by relaxing the requirement of having a dissipation rate involving the whole Lyapunov function $V$.

The explicit estimate of the p-IOS or IOS gain given in the above statement provides a guarantee on the maximal output steady-state error induced by the disturbance $d$. It can also prove useful for possible use with IOS small-gain theorems, such as \cite{jiang2008generalization,Karafyllis:2007kz,Karafyllis:2011fk,KAJIbook11}.

\begin{proof}[of Theorem \ref{theorem-lyap-ios1}] Define the following family of functions parameterized by $s\ge 0$:
\begin{equation} \label{GrindEQ__41_} 
U(x;s):=\frac{1}{2} \left(\left(W(x)-\chi \left(s\right)\right)^{+} \right)^{2},\quad \forall x\in\mR^n. 
\end{equation} 
Then we have that
\begin{align*}
    \nabla U(x;s)f(x,d)=\left(W(x)-\chi \left(s\right)\right)^{+} \nabla W(x)f(x,d).
\end{align*}
In view of \eqref{GrindEQ__40_}, it then holds that
\begin{equation} \label{GrindEQ__42_} 
s\ge \textcolor{teal}| d | \quad \Rightarrow\quad \nabla U(x;s)f(x,d)
\le 0.
\end{equation} 
Let arbitrary $x_{0} \in {\mathbb R}^{n} $ and \textit{$d\in L^{\infty } \left(\mRp ;D\right)$} be given. Using \eqref{GrindEQ__42_} with $s=\left\| d\right\| _{\infty } $, we get that the mapping $t\mapsto U\left(\phi (t,x_{0} ;d);\left\| d\right\| _{\infty } \right)$ is non-increasing on $\mRp$. Consequently, we obtain from \eqref{GrindEQ__41_} that, given any $\tau\geq 0$,
\begin{equation} \label{GrindEQ__43_} 
\hspace{-7pt}
W\left(\phi (t,x_{0} ;d)\right)\le \max \left\{W\left(\phi (\tau ,x_{0} ;d)\right),\chi \left(\left\| d\right\| _{\infty } \right)\right\}, 
\end{equation} 
for all $t\ge \tau$.
Since $W$ is continuous and vanishes at the origin, there exists $b\in \cK_{\infty }$ such that 
\begin{align}\label{eq-14}
W(x)\le b\left(\left|x\right|\right),\quad \forall x\in {\mathbb R}^{n}.
\end{align}
Consequently, we get from \eqref{GrindEQ__43_} with $\tau =0$ that, for all $x_{0} \in {\mathbb R}^{n} $ and all $d\in L^{\infty }\left(\mRp ;D\right)$,
\begin{equation} \label{GrindEQ__44_} 
W\left(\phi (t,x_{0} ;d)\right)\le \max \left\{b\left(\left|x_{0} \right|\right), \chi \left(\left\| d\right\| _{\infty } \right)\right\},
\end{equation} 
for all $t\ge 0$.
Let arbitrary $\varepsilon ,R>0$ be given. Let $\bar W(R):=\max \left\{W(x)\, :\, \left|x\right|\le R\right\}$ and define
\begin{align} 
\tilde{\rho }(\varepsilon ,R)&:=\inf \left\{\, \rho (s)\, :\, s\in[\varepsilon,\varepsilon +\bar W(R)]\, \right\} \label{GrindEQ__45_} \\
T(\varepsilon ,R)&:=\frac{1+\max \left\{\, V(x)\, :\,  \left|x\right|\le R\, \right\}}{\tilde{\rho }(\varepsilon ,R)}.  \label{GrindEQ__46_} 
\end{align} 
Notice that the assumption that $\rho$ is a positive definite function guarantees that $\tilde{\rho }(\varepsilon ,R)$ is positive. We next prove by contradiction that, for all $x_{0} \in {\mathbb R}^{n}$ with $\left|x_{0} \right|\le R$ and all $d\in L^{\infty } \left(\mRp ;D\right)$, 
\begin{equation} \label{GrindEQ__47_}
W\left(\phi (t,x_{0} ;d)\right)\le \max \left\{\varepsilon , \chi \left(\left\| d\right\| _{\infty } \right) \right\}, 
\end{equation}
for all $t\ge T(\varepsilon ,R)$.
To that end, suppose on the contrary that there exists $t\ge T(\varepsilon ,R)$, $d\in L^{\infty } \left(\mRp ;D\right)$ and $x_{0} \in {\mathbb R}^{n} $ with $\left|x_{0} \right|\le R$ for which $W(\phi (t,x_{0} ;d))> \max \left\{\varepsilon , \chi \left(\left\| d\right\| _{\infty } \right) \right\}$. Then we get from \eqref{GrindEQ__43_} that
\begin{equation} \label{GrindEQ__48_} 
\max \left\{\varepsilon , \chi \left(\left\| d\right\| _{\infty } \right) \right\} \le W\left(\phi (t,x_{0} ;d)\right)\le W\left(\phi (\tau ,x_{0} ;d)\right), 
\end{equation} 
for all $\tau \in \left[0,t\right]$.
Using \eqref{GrindEQ__39_}, it follows that
\begin{equation} \label{GrindEQ__49_} 
\frac{d}{d\, \tau } \left(V\left(\phi (\tau ,x_{0} ;d)\right)\right)\le -\rho \left(W\left(\phi (\tau ,x_{0} ;d)\right)\right), 
\end{equation} 
for almost all $\tau \in \left[0,t\right]$.
Since $\left|x_{0} \right|\le R$, it holds that $W\left(x_{0} \right)\le \max \left\{\, W(x)\, :\, \left|x\right|\le R\, \right\}$. Consequently, we obtain from \eqref{GrindEQ__45_}, \eqref{GrindEQ__48_} and \eqref{GrindEQ__49_} that
\begin{equation} \label{GrindEQ__50_} 
V\left(\phi (t,x_{0} ;d)\right)\le V\left(x_{0} \right)-\tilde{\rho }\left(\varepsilon ,R\right) t. 
\end{equation} 
Since $V\left(x_{0} \right)\le \max \left\{V(x)\, :\,  \left|x\right|\le R\right\}$, \eqref{GrindEQ__50_} in conjunction with our assumption that $t\ge T(\varepsilon ,R)$, defined in \eqref{GrindEQ__46_}, leads to the contradiction that $V\left(\phi (t,x_{0} ;d)\right)<0$ and \eqref{GrindEQ__47_} is proved.

Now, define the function $A:\mRp\times\mRp\to\mRp$ as
\[ 
A(r,s):=\left\{\begin{array}{cl}
r& \quad \textrm{if  }r>s \\
0 & \quad \textrm{if  }r \le s
\end{array}\right.,\quad \forall r,s\geq 0,
\]
and the function $\Phi:\mRp\times\mRp\to\mRp$ as 
\begin{align}\label{attempt1}
    \Phi(s,t) :=\sup\Bigl\{A\left(W(\phi(t,x_0;d)),\chi(\|d\|_\infty)\right)\, : \Bigr.
     \Bigl.\, d\in L^{\infty } \left(\mRp ;D\right), \, |x_0|\leq s \Bigr\}, \quad \forall s,t\geq 0.
\end{align}
Let arbitrary $d\in L^{\infty } \left(\mRp ;D\right)$, $x_{0} \in \mathbb R^{n}$ and $t \geq 0$ be given. Notice that if $W(\phi(t,x_0;d))>\chi(\|d\|_\infty)$ then we get from \eqref{GrindEQ__44_} that 
\[
A\left(W(\phi(t,x_0;d)),\chi(\|d\|_\infty)\right)=W(\phi(t,x_0;d))\le b(|x_0|).
\]
On the other hand, if $W(\phi(t,x_0;d)) \le \chi(\|d\|_\infty))$ then
\[ A\left(W(\phi(t,x_0;d)),\chi(\|d\|_\infty)\right)=0 \]
Thus, it holds for all $d\in L^{\infty } \left(\mRp ;D\right)$ and  $x_{0} \in {\mathbb R}^{n}$ that
\[ A\left(W(\phi(t,x_0;d)),\chi(\|d\|_\infty)\right) \le b(|x_0|),\quad \forall t\geq 0. \]
Consequently, definition \eqref{attempt1} gives 
\begin{align}\label{attempt2}
    \Phi(s,t) \le b(s),\quad \forall s,t \geq 0.
\end{align}
Let arbitrary $\varepsilon ,R>0$ be given and consider any $d\in L^{\infty }(\mRp,D)$ and any $x_{0} \in {\mathbb R}^{n}$ with $|x_0|\leq R$. If $W(\phi(t,x_0;d))>\chi(\|d\|_\infty)$ then we get from \eqref{GrindEQ__47_} that 
\[ 
A\left(W(\phi(t,x_0;d)),\chi(\|d\|_\infty)\right)=W(\phi(t,x_0;d))\le \varepsilon,
\]
for all $t\ge T(\varepsilon ,R)$.
On the other hand, if $W(\phi(t,x_0;d)) \le \chi(\|d\|_\infty)$ then we have that 
\[ A\left(W(\phi(t,x_0;d)),\chi(\|d\|_\infty)\right)=0,\quad \forall t\ge T(\varepsilon ,R).  
\]
Combining these two cases, we conclude from \eqref{attempt1} that 
\begin{align*}
    \Phi(R,t) \le \varepsilon,\quad  \forall t\ge T(\varepsilon ,R).
\end{align*}
Since definition \eqref{attempt1} implies that $s\mapsto \Phi(s,t)$ is non-decreasing for any fixed $t\geq 0$, it follows that
\begin{align}\label{attempt4}
    \Phi(s,t) \le \varepsilon,\quad \forall t\ge T(\varepsilon ,R),\,s\in [0,R].
\end{align}
Inequalities \eqref{attempt2} and \eqref{attempt4} show that $\Phi$ fulfills the assumptions of \cite[Lemma 15]{albertini1999continuous}. We conclude that there exits $\tilde\beta\in\cKL$ such that
\begin{align}\label{attempt5}
    \Phi(s,t)\leq \tilde\beta(s,t),\quad \forall s,t\geq 0.
\end{align}
Finally, let arbitrary $d\in L^{\infty }(\mRp,D)$, $x_{0} \in {\mathbb R}^{n}$ and $t\geq 0$ be given. If $W(\phi(t,x_0;d))>\chi(\|d\|_\infty)$ then we get from definition \eqref{attempt1} that $W(\phi(t,x_0;d)) \le \Phi(|x_0|,t)$ and  \eqref{attempt5} ensures that
\begin{align*} 
    W(\phi(t,x_0;d)) \leq \tilde\beta(|x_0|,t).
\end{align*}
Hence, for every $d\in L^{\infty }(\mRp,D)$, $x_{0} \in {\mathbb R}^{n}$ and $t \geq 0$,
\begin{align*} 
    W(\phi(t,x_0;d)) \leq \max \left\{ \tilde\beta(|x_0|,t), \chi \left(\left\| d\right\| _{\infty } \right)\right\} .
\end{align*}
In other words, the system \eqref{GrindEQ__1_} with output $W(x)$ is p-IOS with gain $\chi$.
Using \eqref{GrindEQ__5_}, it follows that 
\begin{align*}
   |y(t,x_0;d)|&\leq \max\left\{a^{-1}\circ\tilde\beta(|x_0|,t)\,,\,a^{-1}\circ\chi(\|d\|_\infty)\right\} \\
   &\leq a^{-1}\circ\tilde\beta(|x_0|,t)+a^{-1}\circ\chi(\|d\|_\infty),
\end{align*}
for all $t\ge 0$. This concludes the proof.
\end{proof}

\vspace{3mm}
The next example shows that Theorem \ref{theorem-lyap-ios1} can indeed prove more flexible than existing Lyapunov conditions for IOS (\cite{SONWANIOS-LYA} or \cite[Section 4.7.1]{KAJIbook11}).

\begin{example}[IOS through Theorem \ref{theorem-lyap-ios1}]  \label{ex_chap_A10_101}
Consider the bi-dimensional system
\begin{subequations}\label{eq_chap_A10_ex_1}
\begin{align}
\dot x_1 &= \frac{x_1}{1+|x_1|}\\
\dot x_2 &= \frac{-x_2 + d}{\sqrt{1+ x_1^2}}, 
\end{align}
\end{subequations}
with output $y=x_2$. Consider the continuously differentiable function defined as $V(x) := \frac{1}{2}x_2^2\sqrt{1+x_1^2}$ for all $x=(x_1,x_2)^\top\in\mR^2$. By letting $f(x,d)$ denote the right-hand side of \eqref{eq_chap_A10_ex_1}, it holds that 
\begin{align*}
\frac{\partial V}{\partial x}(x)f(x, d) &= \frac12 x_2^2\frac{x_1}{\sqrt{1+x_1^2}}\frac{x_1}{1+|x_1|}-x_2^2+x_2d\\
&\le \frac12x_2^2 - x_2^2 +x_2d
 \le -\frac{x_2^2}{4}+d^2,
\end{align*}
where we used the inequality $x_2 d\leq  d^2+x_2^2/4$. Notice that the dissipation rate of $V$ involves only the output $x_2$. Considering the continuously differentiable function $W(x):= x_2^2$, it holds that
\begin{align*}
    W(x)\ge 8d^2 \quad \Rightarrow\quad \frac{\partial V}{\partial x}(x)f(x, d) \le -\frac{1}{8}W(x).
\end{align*}
The derivative of $W$ along the system's solutions reads
\begin{align*}
    \frac{\partial W}{\partial x}(x)f(x, d)=\frac{2(-x_2^2 + x_2d)}{\sqrt{1+x_1^2}},
\end{align*}
from which we get the following relation:
\begin{align*}
    W(x)\ge 8 d^2 \quad \Rightarrow\quad \frac{\partial W}{\partial x}(x)f(x, d) \le 0.
\end{align*}
IOS with linear gain then follows from Theorem \ref{theorem-lyap-ios1}.

In \cite{SONWANIOS-LYA} and \cite[Section 4.7.1]{KAJIbook11}, IOS was shown to hold under the assumption of a dissipation rate involving the Lyapunov function itself, namely
\begin{align}
\underline\alpha(|h(x)|)&\leq \tilde V(x)\leq \overline\alpha(|x|)\label{eq-5}\\
\tilde V(x)\geq \chi(|d|)\quad &\Rightarrow\quad\frac{\partial \tilde V}{\partial x}(x)f(x,d)\leq -\alpha(\textcolor{blue}{\tilde V}(x)),\nonumber
\end{align}
with $\underline\alpha,\overline\alpha,\chi\in\cK_\infty$ and a continuous, positive definite function $\alpha:\mRp\to\mRp$. Finding such a function $\tilde V$ for this example does not seem straightforward due to the presence of the $\frac{1}{\sqrt{1+x_1^2}}$ term in \eqref{eq_chap_A10_ex_1}.

Finally, it is worth mentioning that, as shown in \cite{SONWANIOS-LYA}, IOS guarantees the existence of a function $\bar V$ satisfying a sandwich condition like in \eqref{eq-5} and 
\begin{align}\label{eq-6}
\hspace{-8pt}
\bar V(x)\geq \chi(|d|)\ &\Rightarrow\ \frac{\partial \bar V}{\partial x}(x)f(x,d)\leq -\sigma(\bar V(x),|x|),
\end{align}
with $\chi\in\cK_\infty$ and $\sigma\in\cKL$. For this example, the function $\bar V(x):=x_2^2$ is such a function. Nevertheless, \eqref{eq-5}-\eqref{eq-6} were shown to ensure IOS only if the so-called UBIBS property holds, namely 
\begin{align*}
|x(t;x_0,d)|\leq \max\{\eta_1(|x_0|)\,,\,\eta_2(\|d\|_\infty)\},\quad \forall t\geq 0,
\end{align*}
with $\eta_1,\eta_2\in\cK_\infty$. This condition is clearly not fulfilled here as $|x_1(t;x_0,d)|$ tends to infinity when $x_1(0)\not=0$. 
\end{example}

\section{Application to robust adaptive control}\label{sec-adaptive}

\noindent In this section we illustrate how the results that were given above can be applied to adaptive control. For simplicity, we focus on the following scalar system

\begin{equation} \label{GrindEQ__52_} 
\dot{y}=\theta y\varphi (y)+u+d, 
\end{equation}

\noindent where $y\in {\mathbb R}$ is the plant state, $u\in {\mathbb R}$ is the control input, $\theta \in {\mathbb R}$ is an unknown constant parameter, $d\in {\mathbb R}$ is an unknown disturbance and $\varphi $ is a smooth function. For the scalar system \eqref{GrindEQ__36_}, the book \cite{KAKRbook25} provides a simple, novel direct adaptive control scheme that combines two elements: a nonlinear damping with dynamic gain and a Lyapunov-based deadzone in the update law. This adaptive scheme is called Deadzone-Adapted Disturbance Suppression (DADS) control and is given by
\begin{align}
\hspace{-7pt}
u=-\frac{c}{2} y-\frac{1+e^z}{2} \left(\frac{1+\varphi ^{2} (y)y^{2} }{2a} +c+\frac{\varphi ^{2} (y)}{c} \right)y,  \label{GrindEQ__54_} 
\end{align}
where $z\in\mR$ is a dynamic gain driven by
\begin{align} 
\dot{z}&=\Gamma e^{-z}\left(\frac{1}{2} y^{2} -\varepsilon \right)^{+}.
\label{GrindEQ__53_}
 \end{align}
In this control expression, $\Gamma ,\varepsilon ,c,a>0$ are all positive tuning gains. The deadzone-adapted terminology comes from the term $\left(\frac{1}{2} y^{2} -\varepsilon \right)^{+}$ in \eqref{GrindEQ__53_}, which makes $z$ constant whenever the output norm is below $\sqrt{2\varepsilon}$. 

It is shown in \cite{KAKRbook25} that the DADS adaptive feedback law \eqref{GrindEQ__54_}-\eqref{GrindEQ__53_} guarantees that, for any $\theta \in {\mathbb R}$, any $d\in L^{\infty } \left(\mRp,\mR\right)$, and any initial values $y(0),z(0)\in {\mathbb R}$, the solution of \eqref{GrindEQ__52_} satisfies :
\begin{itemize}
    \item[$i)$] the practical IOS property:
\begin{equation*} 
\left|y(t)\right|\le \left|y(0)\right|e^{-ct/2}+\sqrt{\frac{2a}{c} } \left\| d\right\| _{\infty } +\sqrt{\frac{2a}{c} } \left(\left|\theta \right|-1\right)^{+}, 
\end{equation*} 
\item[$ii)$] the no parameter-drift property:
\begin{equation} \label{GrindEQ__56_} 
z(0)\le z(t)\le \Omega \left(\left\| d\right\| _{\infty } ,\left|\theta \right|,\left|y(0)\right|,e^{z(0)}\right), 
\end{equation} 
for some continuous function $\Omega :\textcolor{blue}{\mRp^3\times \mR_{>0}}\to {\mathbb R}$.
\item[$iii)$]the following output asymptotic gain property:
\begin{align*} 
\limsup_{t\to +\infty } \frac{\left|y(t)\right|^{2} }{2} \le \min \left\{\varepsilon ,a\frac{ d _{\infty }^2 \hspace{-1mm}+\hspace{-1mm}\left(\left(\left|\theta \right|-1-e^{z_\infty}\right)^{+} \right)^{2} }{c\left(1+e^{z_\infty}\right)} \right\},
\end{align*}
where $
d_\infty:=\limsup_{t\to+\infty}{|d(t)|}$,  $z_\infty:=\lim_{t\to+\infty} z(t)$.
\item[$iv)$] the regulation property that at least one of the following holds when $\lim_{t\to +\infty } d(t)=0$ (vanishing disturbance): 
\begin{itemize}
    \item[a)] $\lim_{t\to +\infty} y(t)=0$
    \item[b)] $\left|\theta \right|>1$ and $z(t)\le z_\infty <\ln \left(\left|\theta \right|-1\right)$ for all $t\ge 0$.
\end{itemize}
\end{itemize}

\noindent Property $iii)$ guarantees in particular that the maximal steady-state error for the output $y$ is $\sqrt{2\varepsilon}$, which is independent of the disturbance magnitude $\|d\|_\infty$ and can be reduced at will by picking $\varepsilon$ sufficiently small. 

We next deepen this analysis by considering the alternative output $\tilde{y}:=\left(\left|y\right|-\sqrt{2\varepsilon } \right)^{+} $. More precisely, we address the following two questions:
\begin{enumerate}
\item  In the absence of disturbances ($d\equiv 0$), for what values of the unknown parameter $\theta$ does the system \eqref{GrindEQ__52_} in closed loop with \eqref{GrindEQ__54_}-\eqref{GrindEQ__53_} satisfy the UGAOS property for the output $\tilde{y}$?

\item  For what values of $\theta$ does this closed-loop system satisfy the IOS property for the output $\tilde{y}$?
\end{enumerate}

\noindent To this purpose, we use the following functions:
\begin{align} 
V(y,z)&:=\frac{1}{2}\left(\left(\frac{1}{2} y^{2} -\varepsilon \right)^{+}\right)^{2} +\frac{a}{3\Gamma } \left(\left(\left|\theta \right|-1-e^z\right)^{+} \right)^{3} \label{GrindEQ__59_} \\
W(y)&:=\frac{1}{2} \left(\left(\frac{1}{2} y^{2} -\varepsilon \right)^{+} \right)^{2}.  \label{GrindEQ__60_} 
\end{align} 
These functions satisfy the following relationship:
\begin{equation} \label{GrindEQ__63_} 
V(y,z)\ge W(y)\ge \frac{1}{8} \left|\tilde{y}\right|^{4}.  
\end{equation} 
Moreover, for all $\theta \in {\mathbb R}$ and all $y,z,d\in {\mathbb R}$, the derivative of $V$ along the solutions of the closed-loop system \eqref{GrindEQ__52_}-\eqref{GrindEQ__53_} reads
\begin{align} 
\dot{V}(y,z,d) &=\left(\frac{1}{2} y^{2} -\varepsilon \right)^{+} y^{2} \left(\theta \varphi (y)-\frac{c}{2} \right) \nonumber\\ 
&-\left(\frac{1}{2} y^{2} -\varepsilon \right)^{+} y^2\frac{1+e^z}{2} \left(\frac{1+\varphi ^{2} (y)y^{2} }{2a} +c+\frac{\varphi ^{2} (y)}{c} \right) \label{GrindEQ__61_} \\ &+\left(\frac{1}{2} y^{2} -\varepsilon \right)^{+} \hspace{-2mm}yd-a\left(\left(\left|\theta \right|-1-e^z\right)^{+} \right)^{2} \left(\frac{1}{2} y^{2} -\varepsilon \right)^{+}\hspace{-2mm}. \nonumber
\end{align}
Similarly, for the function $W$:
\begin{align}
\dot{W}(y,z,d) =&\left(\frac{1}{2} y^{2} -\varepsilon \right)^{+} \hspace{-2mm}\left(y^{2} \left(\theta \varphi (y)-\frac{c}{2} \right)+yd\right)\label{GrindEQ__62_}  \\ &-\left(\frac{1}{2} y^{2} -\varepsilon \right)^{+} \hspace{-2mm}y^{2} \frac{1+e^z}{2} \left(\frac{1+\varphi ^{2} (y)y^{2} }{2a} +c+\frac{\varphi ^{2} (y)}{c} \right). \nonumber
\end{align} 
Observe that the following inequalities hold:
\begin{align*}
    yd&\le ad^{2} +\frac{1}{4a} y^{2}\\
    y^{2} \left(\frac{1}{2} y^{2} -\varepsilon \right)^{+} &\ge 2\left(\left(\frac{1}{2} y^{2} -\varepsilon \right)^{+} \right)^{2} =4W(y)
\end{align*}
and that
\begin{align*}
    \theta y^{2} \varphi (y)&\le \left|\theta \right|y^{2} \left|\varphi (y)\right|\\
    &\le \left(\left|\theta \right|-1-e^z\right)^{+} y^{2} \left|\varphi (y)\right|+\left(1+e^z\right)y^{2} \left|\varphi (y)\right| \\ 
    &\le a\left(\left(\left|\theta \right|-1-e^z\right)^{+} \right)^{2} +\frac{\varphi ^{2} (y)}{4a} y^{4}+\frac{c}{2} y^{2} \left(1+e^z\right)
    +\left(1+e^z\right)y^{2} \frac{\varphi ^{2} (y)}{2c}. 
\end{align*}
Hence, we obtain from \eqref{GrindEQ__61_}-\eqref{GrindEQ__62_} that, for all $\theta \in {\mathbb R}$ and all $y,z,d\in {\mathbb R}$,
\begin{align} 
\dot{V}(y,z,d)&\le -\left(\frac{1}{2} y^{2} -\varepsilon \right)^{+}\hspace{-2mm}\left(\frac{c}{2} y^{2} - ad^{2}\right), \label{GrindEQ__64_} 
\\
\dot{V}(y,z,0)&\le -2cW(y)  \label{GrindEQ__65_} 
\\
\dot{W}(y,z,d)&\le \left(\frac{1}{2} y^{2} -\varepsilon \right)^{+} \hspace{-2mm}ad^{2}+\left(\frac{1}{2} y^{2} -\varepsilon \right)^{+}\hspace{-2mm} y^{2} \left(\theta \varphi (y)-c-\frac{cy^{2} +2a}{4ac} \varphi ^{2} (y)\right)\label{GrindEQ__66_}\\
\dot{W}(y,z,0)&\le \frac{1}{2} \left(\frac{1}{2} y^{2} -\varepsilon \right)^{+}  \hspace{-2mm} y^{2} \left(2\theta \varphi (y)-2c-\frac{1+\varphi ^{2} (y)y^{2} }{2a} -\frac{\varphi ^{2} (y)}{c} \right).  \label{GrindEQ__67_}
\end{align} 
In view of inequalities \eqref{GrindEQ__63_}, \eqref{GrindEQ__65_} and \eqref{GrindEQ__67_},
Theorem \ref{thm_1} shows that the disturbance-free closed-loop system \eqref{GrindEQ__52_}-\eqref{GrindEQ__53_} with output $\tilde{y}=\left(\left|y\right|-\sqrt{2\varepsilon } \right)^{+} $ is UGOAS 
provided that there exists a constant $r>0$ such that
\begin{equation} \label{GrindEQ__68_}
\theta \varphi (y)\le c+\frac{1+\varphi ^{2} (y)y^{2} }{4a} +\frac{\varphi ^{2} (y)}{2c}, 
\end{equation} 
for all $|y|\in\left[\sqrt{2\varepsilon},\sqrt{2\left(\varepsilon +r\right)}\right]$. 
Thus, we have shown the following.

\begin{proposition}[Disturbance-free system]\label{prop7}
    If there exists $r>0$ such that \eqref{GrindEQ__68_} holds for all $|y|\in\left[\sqrt{2\varepsilon},\sqrt{2\left(\varepsilon +r\right)}\right]$, then the closed-loop system \eqref{GrindEQ__52_}-\eqref{GrindEQ__53_} with output $\tilde{y}=\left(\left|y\right|-\sqrt{2\varepsilon } \right)^{+} $ satisfies the UGAOS property in the absence of disturbances.
\end{proposition}

Observing that $2\theta \varphi (y)\le \frac{ac}{a+\varepsilon c} \theta ^{2} +\frac{a+\varepsilon c}{ac} \varphi ^{2} (y)$, condition \eqref{GrindEQ__68_} holds automatically when 
\begin{align*}
 \left|\theta \right|\le \sqrt{2+\frac{2\varepsilon c}{a} +\frac{1}{2ac} +\frac{\varepsilon }{2a^{2} } }.   
\end{align*}
However, \eqref{GrindEQ__68_} can hold for many other values of $\theta \in {\mathbb R}$ depending on each specific function $\varphi$. For example, when $\varphi \left(\cdot\right)\equiv 1$ then \eqref{GrindEQ__68_} holds for all $\theta \le c+\frac{1+2\varepsilon }{4a} +\frac{1}{2c} $. It is also worth stressing that, by continuity, \textcolor{blue}{the assumption of Proposition \ref{prop7}} is fulfilled if
\[2\theta \varphi (y)<2c+\frac{1+\varphi ^{2} (y)y^{2} }{2a} +\frac{\varphi ^{2} (y)}{c} , \ \textrm{for} \ y=\pm \sqrt{2\varepsilon }. \]

We now address the second question and consider non-zero disturbances. From \eqref{GrindEQ__64_}, \eqref{GrindEQ__66_} and definition \eqref{GrindEQ__60_} we have that, for every $\lambda \in \left(0,1\right)$ and whenever $W(y)\ge \frac{a^{2} }{2c^{2} \left(1-\lambda \right)^{2} } \left|d\right|^{4}$, the following two bounds hold:
\begin{align} \label{GrindEQ__69_} 
\dot{V}(y,z,d)&\le -\frac{c}{2} \lambda y^{2} \left(\frac{1}{2} y^{2} -\varepsilon \right)^{+} \hspace{-2mm}\le -2c\lambda W(y) \\
\dot{W}(y,z,d)&\le \left(\frac{1}{2} y^{2} -\varepsilon \right)^{+} 
\hspace{-1mm}\Bigl[\theta y^{2} \varphi (y) -cy^2 \Bigr.\label{GrindEQ__70_}\\
& \Bigl. + \, c(1-\lambda)\left(\frac{1}{2} y^{2} -\varepsilon \right)^{+}\hspace{-2mm} 
- \frac{cy^{2} +2a}{4ac} y^{2} \varphi ^{2} (y)\Bigr].\nonumber 
\end{align}
By Theorem \ref{theorem-lyap-ios1}, these two implications lead to the following result.
\begin{proposition}[Disturbed system]\label{prop8}
    If there exists $\lambda \in \left(0,1\right)$ such that
\begin{equation} \label{GrindEQ__71_}
2\theta \varphi (y)\le \left(1+\lambda \right)c+\frac{cy^{2} +2a}{2ac} \varphi ^{2} (y)+\frac{2c\left(1-\lambda \right)\varepsilon }{y^{2} }  , 
\end{equation}
for all $|y|\geq \sqrt{2\varepsilon}$,
then the closed-loop system \eqref{GrindEQ__52_}-\eqref{GrindEQ__53_} with output $\tilde{y}=\left(\left|y\right|-\sqrt{2\varepsilon } \right)^{+} $ satisfies the IOS property.
\end{proposition}

Observing that $2\theta \varphi (y)\le \frac{ac}{\varepsilon c+a} \theta ^{2} +\frac{\varepsilon c+a}{ac} \varphi ^{2} (y)$, condition \eqref{GrindEQ__71_} holds automatically when 
\begin{align*}
    \left|\theta \right|< \sqrt{1+\frac{\varepsilon c}{a} }.
\end{align*}
However, here again, condition \eqref{GrindEQ__71_} can hold for other values of $\theta$, depending on the considered function $\varphi$. For instance, when $\varphi \left(\cdot\right)\equiv 1$, \textcolor{blue}{the assumption of Proposition \ref{prop8} is fulfilled for all} $\theta <c+\frac{1}{2c} $.

IOS being a much stronger property than UGAOS of the corresponding disturbance-free system, it is not surprising that condition  \eqref{GrindEQ__71_} is more demanding than \eqref{GrindEQ__68_}.

\section{Conclusion and perspectives}

We have presented some Lyapunov-based conditions to ensure both URGAOS and IOS based on a dissipation rate that may involve only output terms. We have illustrated the interest of the proposed approach through an academic example and by relying on the recent DADS adaptive control technique. We have also proposed a Lyapunov-based result to establish (non-uniform) asymptotic output stability for systems with inputs, \textcolor{blue}{based on a recent relaxation of Barb\u{a}lat's lemma.}

A point that would deserve clarification is whether condition $ii)$ in Theorem \ref{thm-2} could be imposed only for small values of $W(x)$ (as in condition $i)$): we do not know whether the present condition, requested for all $x$, is really needed or whether it is just a proof artifact. 

More importantly, the results are developed only in the context of finite-dimensional systems. The results in \cite[Section 4]{KACH20} suggest that the proof techniques employed here could fit a more abstract framework, covering time-delay systems and partial differential equations.

\bibliographystyle{plain}
\bibliography{refs}

\end{document}